\documentclass[usenames,dvipsnames, 11pt]{article}
\usepackage[utf8]{inputenc}

\usepackage[margin=2.5cm,a4paper]{geometry} 
\usepackage{multirow,listings,setspace,gnuplottex,latexsym,keyval,ifthen,moreverb,lscape,forest}
\usepackage{pgf,tikz}
\usetikzlibrary{arrows}

\usepackage{amsmath, amssymb, amsthm, bm}
\usepackage{thm-restate}
\usepackage{hyperref}
\usepackage[capitalize]{cleveref}
\usepackage{color}
\usepackage{url}
\usepackage{enumitem}
\usepackage{graphicx}
\usepackage{appendix}
\usepackage{cite}
\usepackage{mathtools}

\newtheorem{thm}{Theorem}[section]
\crefname{thm}{Theorem}{Theorems}

\newtheorem{lem}[thm]{Lemma}
\newtheorem{cor}[thm]{Corollary}
\newtheorem{clm}[thm]{Claim}
\newtheorem{conj}[thm]{Conjecture}

\numberwithin{thm}{section}

\theoremstyle{definition}

\newtheorem{rem}[thm]{Remark}
\newtheorem{ques}[thm]{Question}

\DeclarePairedDelimiter{\parens}{(}{)}
\DeclarePairedDelimiter{\set}{\{}{\}}

\DeclarePairedDelimiter{\floor}{\lfloor}{\rfloor}
\DeclarePairedDelimiter{\ceil}{\lceil}{\rceil}

\DeclarePairedDelimiter{\size}{|}{|}

\usepackage{diagbox}

\newcommand{\eps}{\varepsilon}
\newcommand{\E}{\mathop{\mathbb{E}}}

\newcommand{\qedclaim}{\hfill$\blacksquare$}
\newenvironment{proofclaim}{\removelastskip\penalty55\medskip\noindent{\it Proof of the claim.}}{\medskip\phantom{.}\hfill\qedclaim}

\title{Rainbow connecting $2$-colorings of super-Dirac graphs}
\author{J\'anos Bar\'at\thanks{HUN-REN Alfr\'ed R\'enyi Institute of Mathematics, Budapest, Hungary. Emails:\,
\texttt{barat@mik.uni-pannon.hu}
\texttt{simona@renyi.hu}, \texttt{freschi.andrea@renyi.hu}} \thanks{University of Pannonia, Veszprém, Hungary} \quad Simona Boyadzhiyska\footnotemark[1] \quad Andrea Freschi\footnotemark[1]}

\begin{document}
\maketitle

\begin{abstract}
    Let~$G$ be a graph with minimum degree~$\delta(G)\ge\size{V(G)}/2$.
    Can we color the edges of~$G$ with red and blue so that every pair of non-adjacent vertices is connected by a path consisting of exactly one red edge and one blue edge? 
    We~provide an affirmative answer to this question for a class of graphs that are ``close'' to a complete balanced bipartite graph or the disjoint union of two cliques of the same order. 
    Surprisingly, our methods extend to a much broader class of graphs with minimum degree slightly above $\size{V(G)}/2$.
    Furthermore, we answer an asymptotic version of this question in full, proving that every graph $G$ satisfying $\delta(G)\ge(\size{V(G)}-1)/2$ has a $2$-edge-coloring such that almost all pairs of vertices are connected by a rainbow path.
    In addition, we~propose a number of related open problems.
\end{abstract}

\section{Introduction}

The study of graph connectivity is a central topic in combinatorics and computer science.
A~graph is {\it connected} if there is a path between every pair of vertices.
In this paper, we study a strengthening of this notion called \emph{rainbow connectivity}.
An edge-colored graph $G$ is {\it rainbow connected} if any two vertices of $G$ are connected by a rainbow path, i.e., a path that does not use a color more than once. 
This concept was formally introduced by Chartrand, Johns, McKeon, and Zhang in~\cite{cjmz}. 
One can think of the following natural application to security networks.
Vertices represent agents, edges represent channels through which two agents can communicate, and colors correspond to passwords. 
The further away two agents are in this graph, the more channels the information they exchange needs to cross and the more likely it is to be intercepted. Using the same password to protect two different channels along the way increases the chance for the message to be compromised.    
Information exchanged through rainbow paths is encrypted by multiple, pairwise-distinct passwords.
If the graph is rainbow connected, then every pair of agents can exchange information safely. Naturally, it is desirable for the number of different passwords to be as small as possible.

The minimum number of colors required to color the edges of a graph~$G$ so that it is rainbow connected is called the {\it rainbow connection number} of~$G$ and is denoted by~$\mathrm{rc}(G)$.\footnote{Note that~$\mathrm{rc}(G)$ is well-defined if and only if~$G$ is connected. 
If~$G$ is disconnected, we write $\mathrm{rc}(G)=\infty$.}
Note that, since a rainbow path uses as many colors as its number of edges, $\mathrm{rc}(G)$ is always bounded below by the {\it diameter} of~$G$, i.e., the maximum length of a shortest path between two vertices in~$G$.
It~is easy to see that $\mathrm{rc}(G)=1$ if and only if~$G$ is complete. 
On the other hand, determining whether a graph~$G$ satisfies~$\mathrm{rc}(G)=2$ is NP-complete~\cite{ChakrabortyFMY}, and thus we cannot expect a clean characterization of graphs possessing this property. 
Therefore, researchers have been interested in obtaining bounds on~$\mathrm{rc}(G)$, particularly with respect to specific parameters, including the minimum degree, the edge number and the graph radius (see e.g., \cite{ChandranDRV,ChakrabortyFMY,CaroLRTY,Schiermeyer,KemnitzS,BasavarajuCRR}).

In this paper, we are interested in minimum degree conditions that ensure that $\mathrm{rc}(G)\le2$.
This problem was proposed in~\cite[Problem 4.3]{CaroLRTY}.
For a graph~$G$, we write~$v(G)$ and~$\delta(G)$ to denote the number of vertices and the minimum degree of~$G$, respectively.
Several groups of authors~\cite{CaroLRTY,conj,gim,survey} posed the following question, which is a major open problem in the area. 

\begin{ques}\label{q:dirac}
    Does every graph $G$ with $\delta(G)\geq v(G)/2$ satisfy $\mathrm{rc}(G)\le2$?
\end{ques}

\begin{rem}\label{remark:non-adjcanet}
    Note that a colored edge is a rainbow path connecting its endpoints.
    Therefore, $\mathrm{rc}(G)\le2$ if and only if there is a $2$-edge-coloring of~$G$ such that every non-adjacent pair of vertices is connected by a path consisting of exactly two edges of opposite colors.
\end{rem}

Note that \cref{q:dirac} has a negative answer if we replace~$v(G)/2$ with a lower number.
Indeed, fix an integer $k\ge1$ and let~$H$ be obtained from a cycle $C=v_1v_2\dots v_{4k+1}v_1$ by adding an edge between every pair of vertices at distance at most~$k$ in~$C$.
Note that~$H$ has~$4k+1$ vertices and it is $2k$-regular.
Consider an arbitrary $2$-edge-coloring of~$H$.
The edges in~$H$ corresponding to vertices at distance~$k$ in~$C$ form a cycle~$C'$ on~$4k+1$ vertices.
Since~$C'$ has odd order, there must be two consecutive edges in~$C'$ of the same color, say~$v_1v_{k+1}$ and~$v_{k+1}v_{2k+1}$.
Crucially,~$v_{k+1}$ is the unique common neighbor of~$v_1$ and~$v_{2k+1}$, so~$H$ is not rainbow connected.
Thus, $\mathrm{rc}(H)\ge3$.

\smallskip

\cref{q:dirac} is still wide open.
It is known that $\delta(G)\ge (v(G)-1)/2$ 
implies $\mathrm{rc}(G)\le 3$, provided~$v(G)$ is sufficiently large.
This follows from the following result.

\begin{thm}[Chakraborty, Fisher, Matsliah, Yuster~\cite{ChakrabortyFMY}]\label{thm:ChakrabortyFMY}
    There is an absolute constant $C>0$ such that, if $\delta(G)\ge C\log_2 n$  and $G$ has diameter at most two, then $\mathrm{rc}(G)\le 3$. 
\end{thm}

\begin{rem}
    It is easy to show that, if $\delta(G)\ge (v(G)-1)/2$, then every pair of non-adjacent vertices in~$G$ has at least one common neighbor, and so \cref{thm:ChakrabortyFMY} applies.
\end{rem}

It is also known that, if~$G$ is an $n$-vertex graph with $\delta(G)\geq n/2+\log_2 n$, then $\mathrm{rc}(G)\le2$ ~\cite[Theorem~1.5]{CaroLRTY}.
In fact, the proof of this result shows that a random $2$-edge-coloring of such~$G$ is rainbow connected with high probability.

\begin{thm}[{Caro, Lev, Roddity, Tuza, Yuster~\cite{CaroLRTY}}]\label{prop:random-coloring}
Suppose $G$ is an $n$-vertex graph with $\delta(G)\geq n/2+\log_2 n$. 
Then $\mathrm{rc}(G)\le2$.
\end{thm}
We will present the proof in \cref{sec:random}, for a slight strengthening of this result will play a key role in our arguments.
\medskip

\subsection{Results}
We are mainly concerned with tackling \cref{q:dirac} for some classes of graphs, where we may allow the minimum degree to be slightly higher than in the original formulation.
For convenience, we refer to a graph $G$ satisfying $\delta(G)\geq v(G)/2$ as a \emph{Dirac graph}, after Dirac's well-known result providing a sufficient condition for the existence of a Hamilton cycle~\cite{Dirac}. 
When we wish to emphasize that $G$ has minimum degree strictly exceeding $\lceil v(G)/2\rceil$,
 we say that $G$ is a \emph{super-Dirac graph}. 

Dirac graphs have been studied extensively in extremal combinatorics. Among other properties, it is known that graphs $G$ of minimum degree roughly $v(G)/2$ exhibit a particular trichotomy: they are either `close' to a balanced complete bipartite graph or a disjoint union of two complete graphs on about $v(G)/2$ vertices each, or they have good expansion properties (see~\cite[Lemma~26]{klo} for the precise formulation of this statement). 

For brevity, we call a rainbow connecting $2$-edge-coloring an rc$2$\emph{-coloring}. 
We write~$K_{n,n}$ for the complete bipartite graph with~$n$ vertices in each class, and~$2K_n$ for the disjoint union of two cliques of~$n$ vertices each.
Our first two results state that Dirac graphs that roughly resemble~$K_{n,n}$ or~$2K_n$ admit rc2-colorings.

\begin{thm}[Dirac graphs close to~$K_{n,n}$]\label{thm:Dirac-K_nn}
    Let~$G$ be a graph on~$2n$ vertices with $\delta(G)\ge n$, where $n$ is sufficiently large.
    If there is a partition $V(G)=A\sqcup B$ such that $|A|=|B|=n$ and every vertex has at most~$\sqrt{n}/3$ neighbors in its own part, then~$G$ admits an $\mathrm{rc2}$-coloring.
\end{thm}

\begin{thm}[Dirac graphs close to~$2K_n$]\label{thm:Dirac-2K_n} 
    Let~$G$ be a graph on~$2n$ vertices with $\delta(G)\ge n$, where $n$ is sufficiently large.
    If there is a partition $V(G)=A\sqcup B$ such that $|A|=|B|=n$ and every vertex has at most~$8^{-1}\log_2{n}$ neighbors in the opposite part, then~$G$ admits an $\mathrm{rc2}$-coloring. 
\end{thm}

\begin{rem}\label{remark:eps-close}    
    The assumptions of \cref{thm:Dirac-K_nn} and~\ref{thm:Dirac-2K_n} are more restrictive than the notion of ``$\eps$-close'' to~$K_{n,n}$ or~$2K_{n}$ from~\cite[Lemma~26]{klo}.
    See \cref{sec:conclusion} for further discussion.
\end{rem}

We remark that we do not make a serious effort to optimize the constants in \cref{thm:Dirac-K_nn,thm:Dirac-2K_n}, since we tend to believe that \cref{q:dirac} should have an affirmative answer, and hence the maximum degree conditions are not optimal.
We do not know if our methods can give a better order of magnitude for the bound on the maximum degrees.

\medskip

It is not known whether there exists a constant $c>0$ such that $\delta(G)\ge v(G)/2+c$ implies $\mathrm{rc}(G)\le2$, and this is a compelling problem in itself. 
Perhaps surprisingly, letting~$c$ be a small positive constant already allows us to handle a significantly larger class of graphs than those considered in~\cref{thm:Dirac-K_nn,thm:Dirac-2K_n}.
Note that, in the next theorem, the minimum degree condition is replaced by a lower bound on the sum of the degrees of non-adjacent vertices.
This is a so-called {\it Ore-type} condition, after Ore's generalization of Dirac's theorem~\cite{Ore}. 
In light of \cref{remark:non-adjcanet}, Ore-type conditions are natural in our setting.

\begin{thm}\label{thm:super-Dirac-general}
Let $C\geq 4$. 
Let~$G$ be an $n$-vertex graph such that  $\deg(u) +\deg(v)\geq n-1+C$ for every non-adjacent pair $u,v\in V(G)$, and the vertices can be partitioned into sets $V(G)=A\sqcup B$ such that every non-adjacent pair of vertices $x,y$ on the same side of the bipartition has at least~$n^{8/(2C+1)}$ common neighbors.
If~$n$ is sufficiently large, then~$G$ admits an $\mathrm{rc2}$-coloring.
\end{thm}

The next corollary, which follows immediately from \cref{thm:super-Dirac-general} with~$C=5$, highlights how a slightly stronger minimum degree condition allows us to considerably relax the other assumptions of \cref{thm:Dirac-K_nn,thm:Dirac-2K_n}.

\begin{cor}\label{cor:super-Dirac}
Let~$G$ be a graph on~$2n$ vertices with~$\delta(G)\ge n+2$ and with a partition $V(G)=A\sqcup B$ such that $|A|=|B|=n$.
If~$n$ is sufficiently large, then  $G$ admits an $\mathrm{rc2}$-coloring in each of the following cases: 
\begin{enumerate}[label={\rm(\alph*)}]
    \item  Every vertex has degree at least $n/2 + (1/2)\cdot(2n)^{8/11}$ to the opposite side. \hfill [Close to $K_{n,n}$]\label{super-Dirac-close-to-Kn,n-B}
    \item Every vertex has degree at least $n/2 + (1/2)\cdot(2n)^{8/11}$ to the same side. \hfill [Close to $2K_n$]\label{super-Dirac-close-to-2Kn-B}
\end{enumerate}
\end{cor}

Another natural question is whether every $n$-vertex Dirac graph admits a $2$-edge-coloring such that all but $o(n^2)$ pairs of vertices are connected by a rainbow path.
We provide an affirmative answer to this problem.
In fact, we prove this is the case for graphs slightly below Dirac's minimum degree condition.

\begin{thm}\label{thm:almost}
    Let~$G$ be a graph on~$n$ vertices with $\delta(G)\ge(n-1)/2$.
    There is a $2$-edge-coloring of~$G$ such that all but at most~$o(n^2)$ vertex pairs of~$G$ are connected by a rainbow path.    
\end{thm}

Note that the minimum degree condition $\delta(G)\ge(n-1)/2$ in \cref{thm:almost} is best possible.
Indeed, the $n$-vertex graph consisting of two disjoint cliques of size $\floor{n/2}$ and $\ceil{n/2}$ has minimum degree $\floor{n/2}-1$ and is not even connected.

\paragraph{\texorpdfstring{Organization of the paper.}{Organization of the paper}}

In \cref{sec:random}, we prove a slight strengthening of \cref{prop:random-coloring} (\cref{lem:random_coloring}), which is used in later proofs.
Sections~\ref{sec:K_nn-modifications} and \ref{sec:2K_n-modifications} are devoted to the proofs of \cref{thm:Dirac-K_nn,thm:Dirac-2K_n}, respectively.
In \cref{sec:bad-pairs}, we describe a general framework to attack \cref{q:dirac}.
In particular, we state and prove a technical result (\cref{lem:technical}), which is used to prove \cref{thm:super-Dirac-general,thm:almost}.
Finally, in \cref{sec:conclusion}, we raise various open problems and discuss future research directions.

\paragraph{\texorpdfstring{Notation.}{Notation}} 
We use standard notation for graphs and sets. Let~$G$ be a graph.
For each vertex~$v\in V(G)$, we write~$N_G(v)$ for the set of neighbors of~$v$ in~$G$.
We write $\deg_G(v)=\size{N_G(v)}$ for the degree of~$v$ in~$G$.
For vertex sets~$A,B\subseteq V(G)$, we write~$G[A]$ for the subgraph of~$G$ induced by~$A$, and~$G[A,B]$ for the subgraph induced by the edges with one endpoint in $A$ and one endpoint in~$B$. 
We write~$E_G(A)$ and~$E_G(A,B)$ for the edge sets of~$G[A]$ and $G[A,B]$, respectively.
Set $e(G):=\size{E(G)}$, $e_G(A):=\size{E_G(A))}$, and $e_G(A,B):=\size{E_G(A,B))}$.
We also write $\Delta(G)$, $\Delta_G(A)$, and $\Delta_G(A,B)$ for the maximum degree of~$G$, $G[A]$, and~$G[A,B]$ respectively.
In coloring arguments, we often identify $G$ with its edge set.
When~$G$ is clear from the context, we omit the subscript $G$ (e.g., we write $N(v)$ instead of~$N_G(v)$).
Furthermore, given a set of vertices~$U\subseteq V(G)$, we write $N_U(v)$ for the set of neighbors of~$v$ in~$U$, and $\deg_U(v)$ for the number of neighbors of~$v$ in~$U$, that is, $\deg_U(v)=\size{N_U(v)}$.
Given a set of vertices $S\subseteq V(G)$, we write $G-S$ for the  graph obtained by removing the vertices in~$S$ from~$G$.
Similarly, given a set of edges $T\subseteq E(G)$, we write $G-T$ for the  graph obtained by removing the edges in~$T$ from~$G$.

Given two sets~$X$ and~$Y$, we write $X-Y$ for the set of elements in~$X$ but not in~$Y$.
We write $X\Delta Y$ for the symmetric difference of~$X$ and~$Y$ i.e., the set of elements lying in the union of~$X$ and~$Y$ but not in their intersection.

\section{Random coloring}\label{sec:random}

In this section, we prove a slight strengthening of \cref{prop:random-coloring}, which we use repeatedly throughout our arguments. 
The proof uses a standard probabilistic idea.

\begin{lem}\label{lem:random_coloring}
    Every $n$-vertex graph $H$ has a $2$-edge-coloring such that every pair of vertices $u,v\in V(H)$ satisfying $\size{N(u)\cap N(v)} \ge 2\log_2 n$ is connected by a rainbow path. 
\end{lem}
\begin{proof}
    Color each edge independently and uniformly at random with the colors red and blue. Let $u,v\in V(H)$ be  vertices satisfying $\size{N(u)\cap N(v)} \ge 2\log_2 n$. For a given $w\in N(u)\cap N(v)$, we have $\mathbb P[uwv \text{ is monochromatic}] = \frac12$, and so
    \begin{align*}
        \mathbb P\left[\begin{array}{l}
        uwv \text{ is monochromatic} \\
        \text{for all } w\in N(u)\cap N(v)
        \end{array}\right]  = \parens*{\frac12}^{\size{N(u)\cap N(v)}}\le\parens*{\frac12}^{2\log_2 n}=n^{-2}.    
    \end{align*}
    There are at most $\binom{n}{2}$ such pairs $(u,v)$.
    Since $\binom{n}{2}\cdot n^{-2} <1$, by the union bound every pair of vertices $u,v\in V(H)$ satisfying $\size{N(u)\cap N(v)} \ge 2\log_2 n$ are connected by a rainbow path with positive probability.
    This concludes the proof.
\end{proof}

\section{\texorpdfstring{Dirac graphs close to $K_{n,n}$: proof of \cref{thm:Dirac-K_nn}}{Dirac graphs close to K_{n,n}: proof of Theorem 1.6}}\label{sec:K_nn-modifications}
    Our strategy is to first pre-color some edges to ensure that every pair $(a,b)$ of non-adjacent vertices $a\in A$ and $b\in B$ is connected by a rainbow path.
    We achieve this by coloring the edges in~$G[A]$ and~$G[B]$ red and a small set of edges $F\subseteq E(A,B)$ blue.
    Coloring the remaining edges $E(A,B)-F$ with red and blue, uniformly and independently at random, will take care of non-adjacent pairs within the same part.

    We now define the set~$F\subseteq E(A,B)$. 
    For every pair~$(a,b)$ of non-adjacent vertices $a\in A$ and~$b\in B$, pick an arbitrary common neighbor~$w_{ab}$ of~$a$ and~$b$.
    Note that~$w_{ab}$ exists since $\delta(G)\ge n$ and $v(G)=2n$.
    Furthermore, let $c_{ab}:=a$ if $w_{ab}\in B$ or $c_{ab}:=b$ if $w_{ab}\in A$.
    In particular, the edge~$w_{ab}c_{ab}$ is the edge of the path $aw_{ab}b$ that is incident to both~$A$ and~$B$.
    We let
    $$F:=\{w_{ab}c_{ab}:a\in A,\, b\in B,\, ab\notin E(G)\}\subseteq E(A,B).$$

    Let $d:=\max\{\Delta(A),\Delta(B)\}$; recall that $d\le \sqrt{n}/3$ by assumption.
    Note that every vertex has at most~$d$ neighbors in its part.
    Also, since $\delta(G)\ge n$ and $\size{A}=\size{B}=n$, every vertex has at most $d$ non-neighbors in the opposite part.

    \begin{clm}\label{claim:phaseI}
        Every vertex is incident to at most~$2d^2$ edges in~$F$.    
    \end{clm}
    \begin{proofclaim}
        Fix an arbitrary vertex $v\in V(G)$.
        We upper bound the number of pairs~$(a,b)$ of non-adjacent vertices $a\in A$ and $b\in B$ such that the edge $w_{ab}c_{ab}$ contains~$v$.
        
        First, we bound pairs for which~$v=w_{ab}$.
        In this case, both~$a$ and~$b$ must be adjacent to~$v$.
        Note that~$v$ has at most~$d$ neighbors in its part.
        In turn, each such neighbor has at most~$d$ non-neighbors in the opposite part.
        Therefore, there are at most~$d^2$ choices for~$a$ and~$b$.
        
        Next, we bound pairs for which~$v=c_{ab}$. 
        We must have either $a=v$, if~$v\in A$, or $b=v$, if~$v\in B$. 
        Then we have at most~$d$ choices for the other vertex of the pair, since~$v$ has at most~$d$ non-neighbors in the opposite part.
        
        We conclude that~$v$ is incident to at most~$d^2+d\le2d^2$ edges in~$F$, as required.
    \end{proofclaim}

    We now pre-color all edges of $F$ blue and all edges of $G[A]$ and $G[B]$ red. Notice that, for every pair~$(a,b)$ of non-adjacent vertices $a\in A$ and $b\in B$, the vertices~$a$ and~$b$ are connected by a rainbow path, namely $aw_{ab}b$. 
    It remains to handle non-adjacent pairs consisting of either two vertices from~$A$ or two vertices from~$B$. 
    
    Let~$H$ be the graph obtained from~$G[A,B]$ by removing the edges in~$F$.
    Since every vertex in~$G$ has at least $n-d$ neighbors in the opposite part, \cref{claim:phaseI} implies $\delta(H)\ge n-d-2d^2$.
    In particular, the assumption $d\le \sqrt{n}/3$ implies $\delta(H)\ge 2n/3$.   
    For every pair of vertices $u,v\in A$ (or $u,v\in B$), we have
    \begin{align*}
        \size{N_H(u)\cap N_H(v)}&\geq \deg_H(u) + \deg_H(v) -n\geq  4n/3 -n= n/3\ge2\log_2(2n),
    \end{align*}
    where the last inequality holds for, say, $n\ge 40$.
    Hence, by \cref{lem:random_coloring}, there exists a $2$-edge-coloring of~$H$ such that every pair of vertices in the same part is connected by a rainbow path.
    Thus, the overall $2$-edge-coloring of~$G$ is rainbow connected, as required.\qed

\section{\texorpdfstring{Dirac graphs close to $2K_n$: proof of \cref{thm:Dirac-2K_n}}{Dirac graphs close to 2K_n: proof of Theorem 1.7}}\label{sec:2K_n-modifications}

As in the proof of \cref{thm:Dirac-K_nn}, we construct an rc$2$-coloring with a mix of deterministic and random strategies.
This turns out to be considerably more involved in the present setting.

\smallskip

First, note that every two non-adjacent vertices in~$G$ have at least two common neighbors since $\delta(G)\ge n$ and $v(G)=2n$.
We say a pair of vertices $(a,b)$ with  $a\in A$ and~$b\in B$ is {\it flexible} if they are non-adjacent and they have at least one common neighbor in each of~$A$ and~$B$.
We say a pair of vertices $a\in A$ and~$b\in B$ is {\it inflexible} if they are non-adjacent and all their common neighbors lie in the same part, i.e., either $N(a)\cap N(b)\subseteq A$ or $N(a)\cap N(b)\subseteq B$.

It suffices to show that there exist two sets of edges~$E_f$ and~$E_i$, not necessarily disjoint, from~$E(A)\cup E(B)$ with the following three properties:
\begin{enumerate}[label=(\roman*)]
    \item Every flexible pair has a common neighbor within the edge set $E(A,B)\cup E_f$.\label{case:a}
    \item Every inflexible pair has a common neighbor within the edge set $E(A,B)\cup E_i$.\label{case:b}
    \item Every pair of vertices in~$A$ has at least~$2\log_2{n}$ common neighbors in~$G[A]-(E_f\cup E_i)$, and
    every pair of vertices in~$B$ has at least~$2\log_2{n}$ common neighbors in~$G[B]-(E_f\cup E_i)$.\label{case:c}
\end{enumerate}

Indeed, given such sets~$E_i$ and~$E_f$, we color all edges in~$G[A,B]$ red, all edges in~$E_i\cup E_f$ blue, and all remaining edges with red and blue uniformly and independently at random.
From property (i), it follows that every flexible pair is connected by a rainbow path, using one red edge from~$E[A,B]$ and one blue edge from~$E_f$. 
Similarly, from property (ii), it follows that every inflexible pair is connected by a rainbow path, using one red edge from~$E[A,B]$ and one blue edge from~$E_i$. 
Finally, by \cref{lem:random_coloring}, with high probability every pair of vertices in~$A$ is connected by a rainbow path in $G[A]-(E_f\cup E_i)$, and every pair of vertices in~$B$ is connected by a rainbow path in $G[B]-(E_f\cup E_i)$.
In the rest of the proof, we show that such sets~$E_f$ and~$E_i$ exist.

\smallskip

Let $d:=\Delta(A,B)$, and so $d\le8^{-1}\log_2{n}$ by assumption.
Note that every vertex has at most~$d$ neighbors in the opposite part.
Also, since $\delta(G)\ge n$ and $\size{A}=\size{B}=n$, every vertex has at most~$d$ non-neighbors in its part.

\medskip

\noindent{\bf Constructing~$E_f$: flexible pairs.}
For every flexible pair $(a, b)\in A\times B$, pick arbitrary vertices~$u_{ab}\in N_A(a)\cap N_A(b)$ and $v_{ab}\in N_B(a)\cap N_B(b)$. 
These exist by the definition of flexibility.
In particular, $au_{ab}$ and~$bv_{ab}$ are edges in~$G[A]$ and~$G[B]$, respectively.
Note that a vertex might be picked multiple times when considering different flexible pairs~$(a,b)$.

Next, we show that an edge~$xy$ of~$G[A]$ cannot correspond to~$au_{ab}$ for many flexible pairs~$(a,b)$.
Indeed, if~$a=x$ and~$u_{ab}=y$, then there are at most $\deg_B(y)$ choices for~$b$, since~$y$ and~$b$ must be adjacent;
similarly, if~$a=y$ and~$u_{ab}=x$, then there are at most $\deg_B(x)$ choices for~$b$, since~$x$ and~$b$ must be adjacent.
We conclude that
\begin{align}\label{eq:repeated_edges_1}
|\set{(a,b)\in A\times B:ab\text{ is flexible and } xy=au_{ab}}|\le \deg_B(x)+\deg_B(y)\le 2d.
\end{align}
Similarly, if~$xy$ is an edge of~$G[B]$, then 
\begin{align}\label{eq:repeated_edges_2}
|\set{(a,b)\in A\times B:ab\text{ is flexible and } xy=bv_{ab}}|\le \deg_A(x)+\deg_A(y)\le2d.
\end{align}
We now construct the set~$E_f$ semi-randomly, starting from the empty set and adding edges.
For every flexible pair $(a,b)\in A\times B$, we flip a fair coin independently.
If it lands heads, we add~$au_{ab}$ to~$E_f$; if it lands tails, we add~$bv_{ab}$ to~$E_f$.
In both cases, there will be a path of length two between~$a$ and~$b$ within the edge set $E(A,B)\cup E_f$; this is precisely property~\ref{case:a}.
Note that an edge in~$E(G)$ may be chosen multiple times to be assigned to~$E_f$, namely when it equals~$au_{ab}$ (or~$bv_{ab}$) for different pairs~$(a,b)$.

The next claim is needed to ensure property~\ref{case:c} holds.

\begin{clm}\label{claim:E_f}
    With positive probability, every pair of vertices in~$A$ has at least~$\sqrt{n}/4$ common neighbors in $G[A]-E_f$, and every pair of vertices in~$B$ has at least~$\sqrt{n}/4$ common neighbors in~$G[B]-E_f$.
\end{clm}

\begin{proofclaim}
Crucially, the events $\{e\in E_f:e\in E(G[A])\}$ are pairwise independent: indeed, whether $e\in E_f$ or $e\notin E_f$ depends only on the coins flipped for the pairs~$(a,b)$ such that $e=au_{ab}$, which are independent from all other coin flips.
Similarly, the events $\{e\in E_f:e\in E(G[B])\}$ are pairwise independent.

For every edge~$e$ in~$G[A]$, we have $\mathbb P[e\notin E_f]=2^{-\ell}$, where~$\ell$ is the number of flexible pairs $(a,b)\in A\times B$ such that $e=au_{ab}$.
By~\eqref{eq:repeated_edges_1}, we conclude that 
$$\mathbb P[e\notin E_f]\ge2^{-2d}.$$
Using~\eqref{eq:repeated_edges_2}, we can prove the same bound holds for an edge in~$e$ in~$G[B]$.
Now, for every pair of vertices~$x,y$ lying on the same side $C\in\{A,B\}$, we have 
\begin{align*}
|N_C(x)\cap N_C(y)|=|N_C(x)|+|N_C(y)|-|N_C(x)\cup N_C(y)|\ge 2(n-d)-n=n-2d\ge \frac{n}{2},
\end{align*}
where the last step holds since $d\le8^{-1}\log_2 n$.
Let~$P_{xy}$ denote the number of paths of the form~$xzy$ such that~$z\in C$ and $xz,zy\notin E_f$.
The probability that a fixed path~$xzy$ does not contain edges from~$E_f$ is
$$\mathbb P[xz\notin E_f \text{ and } zy\notin E_f]=\mathbb P[xz\notin E_f]\cdot \mathbb P[zy\notin E_f]\ge 2^{-4d},$$
where the  equality follows from the independence of the events (as~$xz$ and~$zy$ both lie in~$G[C]$).
In particular, we have 
$$\mathbb E[P_{xy}]\ge|N_C(x)\cap N_C(y)|\cdot 2^{-4d}\ge n\cdot2^{-4d-1}\ge \sqrt{n}/2.$$ 
Since the events $\{xz\notin E_f \text{ and } zy\notin E_f\}$ are pairwise independent (for fixed endpoints~$x$ and~$y$ and varying $z$), we can apply the well-known Chernoff bound to obtain
\begin{align*}
\mathbb P[P_{xy}\le \sqrt{n}/4]&\le \mathbb P[P_{xy}\le 2^{-1}\mathbb E[P_{xy}]]\le\exp(-\mathbb E[P_{xy}]/8)\le\exp(-\sqrt{n}/16)\le n^{-2},
\end{align*}
where the last inequality holds for large~$n$ (say, $n\ge 3\cdot10^5$).
Since there are~$n(n-1)$ possible pairs~$x,y$ (namely,~$\binom{n}{2}$ pairs in~$A$ and $\binom{n}{2}$ pairs in~$B$), the union bound implies that, with positive probability, we have~$P_{xy}\ge \sqrt{n}/4$ for every pair~$x,y$ lying in the same part, as required.  
\end{proofclaim}

\medskip

\noindent{\bf Constructing~$E_i$: inflexible pairs.}
We first show that every vertex is part of few inflexible pairs.
Note that the proof of the following claim (and of \cref{claim:E_i}) closely resembles the analysis in the proof of \cref{claim:phaseI} from \cref{sec:K_nn-modifications}.

\begin{clm}\label{claim:inflexible-pairs}
    Every vertex in~$G$ is part of at most~$2d^2$ inflexible pairs.
\end{clm}

\begin{proofclaim}
    Fix an arbitrary vertex $a\in V(G)$.
    Without loss of generality, we may assume~$a\in A$.
    Recall that~$a$ forms an inflexible pair with $b\in B$ if and only if~$a$ and~$b$ are non-adjacent and either~$N(a)\cap N(b)\subseteq A$ or $N(a)\cap N(b)\subseteq B$.
    We upper bound the number of possible choices of $b$.
    
    First, we count the possible $b$ such that $N(a)\cap N(b)\subseteq A$.
    Note that~$a$ has a neighbor~$x$ in~$B$, since $\delta(G)\ge n$ and $\size{A\setminus\set{a}}=n-1$.
    In turn, $x$ has at least~$n-d$ neighbors in~$B$.
    These neighbors of~$x$ cannot play the role of~$b$, since they have $x\in B$ as a common neighbor with~$a$.
    Thus, there are at most~$d$ choices for~$b$.
    
    Next, we count those~$b$ for which $N(a)\cap N(b)\subseteq B$.
    Note that any such~$b$ has at least one neighbor $y$ in~$A$, since $\delta(G)\ge n$ and $\size{B\setminus\set{b}}=n-1$.
    In turn, $y$ cannot be a neighbor of~$a$.
    Since~$a$ has at most $d$ non-neighbors in~$A$, there are at most~$d$ choices for~$y$.
    Since~$y$ has at most~$d$ neighbors in~$B$, for each fixed~$y$ there are at most~$d$ choices for~$b$.
    We conclude there are at most~$d^2$ choices for~$b$.    
     
    It follows that there are at most $d+d^2\le 2d^2$ inflexible pairs containing~$a$, as required.
\end{proofclaim}

    For every inflexible pair $a\in A$ and $b\in B$, pick a common neighbor~$w_{ab}$ of~$a$ and~$b$.
    We let~$e_{ab}$ be either the edge~$aw_{ab}$, if $w_{ab}\in A$, or the edge~$bw_{ab}$, if $w_{ab}\in B$.
    In particular, the edge~$e_{ab}$ lies completely in either~$A$ or~$B$.
    
    We let $E_i:=\{e_{ab}:a\in A,\, b\in B,\,\text{ $(a,b)$ is inflexible}\}$. 
    Notice that, for every inflexible pair~$(a,b)\in A\times B$, the vertices~$a$ and~$b$ have a common neighbor within the edge set~$E(A,B)\cup E_i$, namely~$w_{ab}$. 
    This is exactly property~\ref{case:b}.

    The next claim is used to ensure property~\ref{case:c} holds.
    Its proof relies on \cref{claim:inflexible-pairs}.

    \begin{clm}\label{claim:E_i}
        Every vertex is incident to at most~$4d^3$ edges in~$E_i$.    
    \end{clm}
    \begin{proofclaim}
        Fix an arbitrary vertex $v\in V(G)$.
        We upper bound the number of inflexible pairs $(a,b)\in A\times B$ such that~$e_{ab}$ is incident to~$v$.
        Note that either $e_{ab}=aw_{ab}$ or $e_{ab}=bw_{ab}$.
        
        First, we bound pairs for which~$v=w_{ab}$.
        In this case, both~$a$ and~$b$ must be incident to~$v$.
        Note that~$v$ has at most~$d$ neighbors in the opposite part.
        In turn, each such neighbor is part of at most~$2d^2$ inflexible pairs by \cref{claim:inflexible-pairs}.
        Therefore, there are at most~$2d^3$ choices for~$a$ and~$b$.
        
        Next, we bound pairs for which~$v\neq w_{ab}$.
        We must have either $a = v$, if $v\in A$, or $b = v$, if~$v\in B$.
        Then we have at most~$2d^2$ choices for the other vertex of the inflexible pair, since~$v$ is part of at most~$2d^2$ inflexible pairs (by \cref{claim:inflexible-pairs}).

        We conclude that~$v$ is incident to at most~$2d^3+2d^2\le4d^3$ edges in~$E_i$, as required.    
    \end{proofclaim}

    \medskip

    \noindent{\bf Ensuring property\ref{case:a},\ref{case:b}, and~\ref{case:c}} 
    Property~\ref{case:b} holds by the definition of~$E_i$.
    By \cref{claim:E_f}, there exists a choice of~$E_f$ such that property~\ref{case:a} holds, every pair of vertices in~$A$ has at least~$\sqrt{n}/4$ common neighbors in $G[A]-E_f$, and every pair of vertices in~$B$ has at least~$\sqrt{n}/4$ common neighbors in $G[B]-E_f$.

    By \cref{claim:E_i}, each vertex in~$A$ is incident to at most~$4d^3$ edges in~$E_i$.
    It follows that every pair of vertices~$x,y\in A$ has at most~$8d^3$ common neighbors~$w$ such that~$xw\in E_i$ or~$yw\in E_i$.
    Therefore, for large $n$, every pair of vertices in~$A$ has at least~$\sqrt{n}/4-8d^3\ge2\log_2{n}$ common neighbors in~$G[A]-(E_f\cup E_i)$, since $d\le8^{-1}\log_2{n}$ and~$n$ is sufficiently large (say, $n\ge 3\cdot10^5$).
    Similarly, every pair of vertices in~$B$ has at least~$2\log_2{n}$ common neighbors in $G[B]-(E_f\cup E_i)$.
    Thus, property~\ref{case:c} holds, and so the proof is complete.\qed

\section{\texorpdfstring{Bad-pairs graph: proofs of \cref{thm:super-Dirac-general} and~\ref{thm:almost}}{Bad-pairs graph: proofs of Theorem 1.9 and 1.11}}\label{sec:bad-pairs}
In this section, we introduce a general framework to attack \cref{q:dirac}, and use it to prove \cref{thm:super-Dirac-general,thm:almost}.
We first introduce some notation.

Given a graph~$G$ and $t\ge0$, we say a pair of non-adjacent vertices $x,y\in V(G)$ is \emph{$t$-bad} if~$\size{N(x)\cap N(y)}\leq t$; if a non-adjacent pair is not $t$-bad, then it is \emph{$t$-good}. 
We define the \emph{$t$-bad-pairs graph} $B(G,t)$ to be the graph on vertex set $V(G)$, where a pair of vertices $(x,y)$ forms an edge if and only if~$(x,y)$ is a $t$-bad pair in~$G$. 

The general approach we take in this section is the following.
Given an $n$-vertex Dirac graph~$G$, we split the edge set into two disjoint sets~$E_1$ and~$E_2$ via a random procedure.
Roughly speaking, the random split ensures that, with high probability, a $t$-good pair in~$G$ has $\Omega(\log n)$ common neighbors in the graph $(V(G),E_1)$, and a $t$-bad pair has at least one common neighbor in the graph $(V(G),E_2)$; the formal statement is given in \cref{lem:technical}.
Our strategy is to color $E_1\cup E_2$ so that every $t$-good pair is rainbow connected in $(V(G),E_1)$ and every $t$-bad pair is rainbow connected in $(V(G),E_2)$.
We achieve this by coloring~$E_1$ randomly (using \cref{lem:random_coloring}), whereas~$E_2$ is colored in a deterministic fashion.
For the latter, we need to consider the structure of the bad-pairs graph~$B(G,t)$, which turns out to be ``almost'' bipartite (see e.g. \cref{lemma:bad-graph-girth}). 

The rest of the section is organized as follows.
We first state and prove a technical lemma (\cref{lem:technical}) that formalizes the aforementioned random splitting procedure of the edge set of a Dirac graph.
We combine \cref{lem:technical} with additional ideas to prove \cref{thm:super-Dirac-general} and \cref{thm:almost} in the second and third subsection, respectively. 

\subsection{A technical lemma: splitting the edge set of a Dirac graph}\label{sec:technical}

The following lemma states that the edge set of a graph satisfying an Ore-type condition can be randomly partitioned into sets~$E_1$ and~$E_2$ so that $t$-good pairs are likely to have many neighbors in $(V(G),E_1)$, whereas $t$-bad pairs are likely to have at least one common neighbor in $(V(G),E_2)$.

\begin{lem}\label{lem:technical} 
    Let $C\ge 0$ and $t> 0$.
    Let $G$ be an $n$-vertex graph such that  $\deg(u) +\deg(v)\geq n-1+C$ for every non-adjacent pair $u,v\in V(G)$.
    Let $p:=\sqrt{30\log_2(n)/t}$ and suppose $p < 1$.
    Partition the edges of $G$ into two sets $E_1\cup E_2$ by  assigning an edge to~$E_1$ with probability~$p$ and to~$E_2$ with probability $1-p$, independently of all other edges.
    The following hold:
    \begin{enumerate}[label=(\roman*)]
        \item The probability that every $t$-good pair has at least $2\log_2 n$ common neighbors in $(V(G),E_1)$ is at least~$1-1/n$.\label{case:good} 
        \item If $(u,v)$ is a $t$-bad pair, the probability that~$u$ and~$v$ have no common neighbor in $(V(G),E_2)$ is at most $(120\log_2(n)/t)^{\frac{C+1}{2}}$.\label{case:bad} 
    \end{enumerate}
\end{lem}

\begin{rem}
    Note that, if the probability in \cref{lem:technical}~\ref{case:bad} is~$o(n^{-2})$, then the union bound implies that {\it every} $t$-bad pair has at least one common neighbor in $(V(G),E_2)$ with probability tending to one. 
    Since~$t\le n$, for this probability to be~$o(n^{-2})$, we need~$C\ge4$.
    This is the reason we need a surplus in the minimum degree in  \cref{thm:super-Dirac-general}.
\end{rem}

\begin{proof}[Proof of \cref{lem:technical}]
First we prove~\ref{case:good}. 
Let $x,y\in V(G)$ be a $t$-good pair, that is, $x$ and $y$ share more than $t$ common neighbors.
The probability that a common neighbor $z\in N(x)\cap N(y)$ remains a common neighbor in $(V(G),E_1)$ is exactly $p^2$. We then have 
\begin{align*}
    \E [\size{\set{z\in N(x)\cap N(y)\,:\, xz,yz\in E_1}}] \geq p^2t \ge 30\log_2 n. 
\end{align*}
Then, by the Chernoff bound, the probability that $\size{\set{z\in N(x)\cap N(y)\,:\, xz,yz\in E_1}}\le 2\log_2 n$ is at most $\exp(-30\log_2 n/8) \le 1/n^3$. 
Thus, taking a union bound over all  $t$-good pairs, of which there are at most $\binom{n}{2}$, we find that with probability at least $1-1/n$ every such pair has more than~$2\log_2n$ common neighbors in $(V(G),E_1)$.

Next, we prove~\ref{case:bad}. 
Let $u,v\in V(G)$ be a $t$-bad pair and $w\in N(u)\cap N(v)$ be a common neighbor. The probability that at least one of~$uw$ and~$vw$ is not in $(V(G),E_2)$ is precisely $1-(1-p)^2 = 2p-p^2\leq 2p$. Thus, the probability that, for every common neighbor~$w$ of~$u$ and~$v$, at least one of~$uw$ and~$vw$ is not in~$E_2$, is bounded above by $(2p)^{\size{N(u)\cap N(v)}}$. 
Since~$u$ and~$v$ are not adjacent, they have at least 
$$\deg(u)+\deg(v)-(n-2)\geq n-1+C-(n-2) =C+1$$ common neighbors in $G$.
Thus, the above probability is bounded above by\footnote{Here we use monotonicity and $2p\le 1$. If $2p> 1$, then the bound is trivial, since a probability is always at most 1.} $(2p)^{C+1}=$\\
$(120\log_2(n)/t)^{\frac{C+1}{2}}$, as required.
\end{proof}

\subsection{\texorpdfstring{Proof of \cref{thm:super-Dirac-general}}{Proof of Theorem 1.9}}

In this subsection, we prove \cref{thm:super-Dirac-general} using \cref{lem:technical}.
The following theorem is equivalent to \cref{thm:super-Dirac-general} but is phrased using the terminology of this section. 

\begin{thm}\label{thm:bipartite-bad-pairs} 
    Let~$C\ge4$.
    Let $G$ be an $n$-vertex graph such that  $\deg(u) +\deg(v)\geq n-1+C$ for every non-adjacent pair $u,v\in V(G)$.
    If $n$ is sufficiently large and $B(G,n^{8/(2C+1)})$ is bipartite, then~$G$ admits an $\mathrm{rc}2$-coloring.
\end{thm}

\begin{proof}
    Let $t:=n^{8/(2C+1)}$ and $p := \sqrt{30\log_2(n)/t}$.
    Note that $p < 1$ for~$n$ sufficiently large.
    Partition the edges of $G$ into sets $E_1\cup E_2$, where an edge is assigned to~$E_1$ with probability~$p$ and to~$E_2$ with probability $1-p$, independently of all other edges.
    By \cref{lem:technical}~\ref{case:bad} and the union bound, the probability that there is some $t$-bad pair with no common neighbor in~$(V(G),E_2)$ is at most
    $$n^2\cdot(120\log_2(n)/t)^{\frac{C+1}{2}}
    =(120\log_2{n})^{\frac{C+1}{2}}\cdot n^{-2/(2C+1)}\le 1/2,$$
    where the last inequality holds since~$n$ is sufficiently large.
    By this and \cref{lem:technical}~\ref{case:good}, it follows that there exists a choice of $E_1\cup E_2$ such that every $t$-good pair has at least~$2\log_2n$ common neighbors in $(V(G),E_1)$ and every $t$-bad pair has at least one common neighbor in $(V(G),E_2)$.  
    Fix such a partition. 
    
    It remains to color the edges of~$G$. 
    By \cref{lem:random_coloring}, there is a $2$-edge-coloring of $(V(G),E_1)$ such that every $t$-good pair is connected by a rainbow path. 
    Let~$U,W$ be a bipartition of the bad-pairs graph~$B(G,t)$. 
    We color all edges in $G[U,W]\cap E_2$ red and all other edges in~$E_2$ blue. 
    Every $t$-bad pair is of the form $a\in U$ and $b\in W$, since~$B(G,t)$ is bipartite.
    We know there is a common neighbor~$w$ of~$a$ and~$b$ such that $aw,bw\in E_2$.
    Without loss of generality, we may assume that~$w\in W$. 
    Then~$aw$ is red, while~$wb$ is blue, giving the required rainbow path.
\end{proof}

\subsection{\texorpdfstring{Proof of \cref{thm:almost}}{Proof of Theorem 1.11}}

In this subsection, we prove \cref{thm:almost} using \cref{lem:technical}.
We start with a general fact about bad-pairs graphs that might be of independent interest.
The {\it odd girth} of a graph is the length of its shortest odd cycle; it is defined to be infinity if the graph is bipartite.
The next lemma establishes a lower bound on the odd girth
of~$B(G,t)$ for a graph~$G$ with the assumed minimum degree.

\begin{lem}\label{lemma:bad-graph-girth}
    For any $n$-vertex graph $G$ with $\delta(G)\ge (n-1)/2$ and $t\ge 1$, the graph~$B(G,t)$ has odd girth at least $n/(2t+1)$. 
\end{lem}

\begin{proof}
    Suppose $v_1v_2\dots v_\ell v_1$ is an odd cycle in~$B(G,t)$ for some odd~$\ell\geq 3$.
    \begin{clm}
        For every $i\in[\ell-2]$, we have $\size{N(v_i)\Delta N(v_{i+2})}\le 4t+2$.
    \end{clm}
    \begin{proof}
        Since $(v_i,v_{i+1})$ and $(v_{i+1},v_{i+2})$ are $t$-bad, there are at least $\size{N(v_{i+1})}-2t$ vertices in~$N(v_{i+1})$ that belong to neither~$N(v_{i})$ nor~$N(v_{i+2})$.  
        Thus, we have
        $$\size{N(v_i)\cup N(v_{i+2})}\le \size{V(G)}-\size{N(v_{i+1})}+2t\le \tfrac{n+1}{2}+2t,$$
        where the last inequality follows from the fact that $\delta(G)\ge (n-1)/2$.
        Hence,
        $$\size{N(v_i)\Delta N(v_{i+2})}=2\size{N(v_i)\cup N(v_{i+2})}-\size{N(v_i)}-\size{N(v_{i+2})}\le 2\cdot\parens*{\tfrac{n+1}{2}+2t}-\tfrac{n-1}{2}-\tfrac{n-1}{2}=4t+2.$$
    \end{proof}
    Note that any sets $U,V,W$ satisfy $U\Delta V \subseteq (U\Delta W) \cup (W\Delta V)$. Using this triangle inequality, and the previous claim, we obtain
    $$\size{N(v_1)\Delta N(v_\ell)}\le\sum_{k=1}^{(\ell-1)/2}\size{N(v_{2k-1})\Delta N(v_{2k+1})}\le \frac{(\ell-1)(4t+2)}{2} = (2t+1)(\ell-1).$$
    Without loss of generality, we may assume $\size{N(v_1)\setminus N(v_\ell)}\le \size{N(v_1)\Delta N(v_\ell)}/2$.
    Hence,
    $$\size{N(v_1)\cap N(v_\ell)}= \size{N(v_1)}-\size{N(v_1)\setminus N(v_\ell)}\ge \frac{n-1}{2}-\frac{(2t+1)(\ell-1)}{2}.$$
    Since $(v_1,v_\ell)$ is $t$-bad, we have $\size{N(v_1)\cap N(v_\ell)}\le t$.
    Rearranging yields~$\ell\ge n/(2t+1)$, as required.
\end{proof}

\begin{proof}[Proof of \cref{thm:almost}]

    Let $t:=(\log_2 n)^2$ and $p := \sqrt{30\log_2(n)/t} = \sqrt{30/\log_2n}$.
    Note that $p < 1$ for~$n$ sufficiently large.
    Partition the edges of $G$ into sets $E_1\cup E_2$, where an edge is assigned to~$E_1$ with probability~$p$ and to~$E_2$ with probability~$1-p$, independently of all other edges.

    Let~$F$ be the set of $t$-bad pairs with no common neighbor in $(V(G),E_2)$.
    By \cref{lem:technical}~\ref{case:bad} applied with $C=0$, we have
    $$\mathbb E[\size{F}]\le\frac{n^2}{2}\cdot(120\log_2(n)/t)^{\frac{1}{2}}=o(n^2).$$
    By Markov's inequality, we have $\mathbb P[\size{F} \ge 2\mathbb E[\size{F}]]\le 1/2$.
    In particular, $\mathbb P[\size{F} < 2\mathbb E[\size{F}]]\ge 1/2$ and $2\mathbb{E}[\size{F}] =o(n^2)$.
    Combining this with \cref{lem:technical}~\ref{case:good}, we conclude that there exists a choice of $E_1\cup E_2$ such that every $t$-good pair has at least~$2\log_2n$ common neighbors in $(V(G),E_1)$, and all but $o(n^2)$ $t$-bad pairs have a common neighbor in $(V(G),E_2)$.  
    Fix such a partition.

    By \cref{lemma:bad-graph-girth}, the graph $B:=B(G,t)$ has odd girth at least $n/(2t+1)$.
    A classic result of Andr{\'a}sfai, Erd{\H o}s, and S{\'o}s~\cite{AndrasfaiES} states that a non-bipartite $n$-vertex graph with odd girth at least~$\ell$ has a vertex of degree at most $2n/\ell$.
    Now, we iteratively remove the edges incident to vertices of minimum positive degree from~$B$ until the remaining graph~$B'$ is bipartite, say with vertex classes~$U$ and~$W$.
    By the result in~\cite{AndrasfaiES}, it follows that we have removed~$o(n^2)$ edges from~$B$.

    It remains to color the edges of~$G$. 
    By \cref{lem:random_coloring}, there is a $2$-edge-coloring of $(V(G),E_1)$ such that every $t$-good pair is connected by a rainbow path. 
    Finally, we color the edges in~$G[U,W]\cap E_2$ red and the remaining edges in~$E_2$ blue. 
    Note that all $t$-bad pairs with at least one common neighbor in~$(V(G),E_2)$ and corresponding to an edge in~$B'$, are connected by a rainbow path in~$(V(G),E_2)$.
    This concludes the proof, since there are~$o(n^2)$ $t$-bad pairs that  do not have a common neighbor in in~$(V(G),E_2)$ or correspond to an edge in $E(B)\setminus E(B')$.
\end{proof}

\section{Concluding remarks}\label{sec:conclusion}

In this paper, we provided an affirmative answer to \cref{q:dirac} for a particular type of Dirac graphs; we also showed that allowing a small surplus in the minimum degree condition allows us to handle a much broader class. As a first step towards answering the original question, it would help to investigate whether allowing a slightly larger minimum degree guarantees the existence of an rc2-coloring.
We believe \cref{thm:super-Dirac-general} is evidence that the following holds.

\begin{conj}
    There exists a universal constant~$C>0$ such that, if~$G$ is a graph with~$\delta(G)\ge v(G)/2+C$, then $\mathrm{rc}(G) \le 2$.
\end{conj}

In most of our proofs, we are looking at the common neighbors of two non-adjacent vertices. 
In particular, we saw in \cref{thm:super-Dirac-general} that an Ore-type condition on the degrees of non-adjacent vertices is enough to guarantee the existence of an rc2-coloring. 
Dong and Li~\cite{dl} initiated this Ore-type line of study in the context of rainbow connections. Let $\sigma_2$ denote $\min\{\mathrm{deg}(v)+\mathrm{deg}(u)\}$ over all non-adjacent vertices $u,v$ in a graph $G$. They proved the following.
\begin{thm}[Dong and Li~\cite{dl}] For a connected $n$-vertex graph~$G$, we have $\mathrm{rc}(G)\le \frac{6n}{\sigma_2+2}+8$.
\end{thm}
In light of this result and our work, it is natural to formulate the following strengthening of \cref{q:dirac}.

\begin{ques}
    Is it true that every $n$-vertex graph $G$ in which every pair $u,v$ of non-adjacent vertices satisfies $\deg(u)+\deg(v) \geq n$ has an rc2-coloring?
\end{ques}

It would be particularly interesting to show that the two questions have different answers. 
We remark that $n$ is a sharp lower bound here, as demonstrated by the construction described after \cref{remark:non-adjcanet}.

\medskip

In light of \cref{remark:eps-close}, a natural question is whether we can strengthen \cref{thm:Dirac-K_nn,thm:Dirac-2K_n} by assuming that~$G$ is ``close'' to~$K_{n,n}$ or~$2K_n$ in the sense meant in~\cite[Lemma~26]{klo}. Given~$\eps>0$, we say that a $2n$-vertex graph $G$ is \emph{$\eps$-close to $K_{n,n}$} if there exists a subset $A\subseteq V(G)$ with $\size{A}=n$ and $e(A) < \eps n^2$. Similarly, $G$ is \emph{$\eps$-close to $2K_{n}$} if there is a partition $V(G) = A\sqcup B$ with $\size{A} = n$ and $e(A,B) < \eps n^2$. 

\begin{ques}\label{q:eps-close}
    Fix~$\eps>0$ and let~$n$ be sufficiently large.
    If $G$ is a $2n$-vertex Dirac graph that is either $\eps$-close to~$K_{n,n}$ or to~$2K_n$, does $\mathrm{rc}(G)=2$ hold?
\end{ques}

Additionally, it would be very interesting to make progress for so-called {\it robust expanders}, which are the third class of graphs in the trichotomy from~\cite[Lemma~26]{klo}. 
Given $0<\nu\leq \tau<1$, an $n$-vertex graph $G$ is a \emph{$(\nu,\tau)$-robust expander} if, for every subset $S\subseteq V(G)$ with $\tau n\leq \size{S} \leq (1-\tau)n$, there are at least $\size{S}+\nu n$ vertices $w$ such that $\deg_S(w)\geq \nu n$.

\begin{ques}
    If $G$ is a Dirac graph that is a $(\nu,\tau)$-robust expander, for some suitable~$\nu$ and~$\tau$, does $\mathrm{rc}(G)=2$ hold?
\end{ques}

Some natural classes of (super-)Dirac graphs to investigate next are those arising from groups, that is, Cayley graphs and more specifically circulant graphs (Cayley graphs of the cyclic group~$\mathbb{Z}_n$). We wonder if the extra structure of these graphs makes it easier to prove the existence of an rc2-coloring in such graphs. 

\begin{ques}
    Does every Dirac circulant (or Cayley) graph have an rc2-coloring?
\end{ques}

One might hope that a simple ``generator-based'' coloring will work, that is, assigning the same color to all edges corresponding to the same generator. However, consider the graph on vertex set $\mathbb Z_{4k}$, where $ij$ is an edge whenever $\size{i-j}\leq k \pmod{4k}$; it is easy to see that, if we decide the color of an edge $ij$ solely based on $\size{i-j}$, then we can never guarantee a rainbow path between antipodal vertices (e.g., $0$ and $2k$). However, some computational experiments we carried out for small groups suggest that increasing the minimum degree slightly (e.g., by a small constant) might guarantee the existence of such a generator-based rc2-coloring.  

\begin{ques}
    Does every circulant (or Cayley) graph $G$ of minimum degree $v(G)/2+C$ for some constant $C$ have a generator-based rc2-coloring?
\end{ques}

Of course, Cayley graphs belong to the more general class of vertex-transitive graphs, which is in itself an interesting case to tackle.

\medskip

Finally, recall that \cref{thm:ChakrabortyFMY} implies that if $\delta(G)\ge (v(G)-1)/2$, and~$v(G)$ is sufficiently large, then $\mathrm{rc}(G) \le 3$. 
Notice that in this case we allow rainbow paths of length at most three. We wonder if the following stronger variant holds.

\begin{ques}\label{q:strong-rainbow}
    Let $G$ be a graph with $\delta(G)\ge (v(G)-1)/2$. Is it possible to 3-color the edges of $G$ so that every non-adjacent pair of vertices is connected by a non-monochromatic path of length two? 
\end{ques}

This is again a special case of a previously studied problem, that of the \emph{strong rainbow connection number}. Formally, this parameter, denoted $\mathrm{src}(G)$, captures the smallest number of colors   needed to color the edges of $G$ so that every pair of distinct vertices is connected by a rainbow geodesic (a shortest path).
Using this terminology, \cref{q:strong-rainbow} asks whether $\delta(G)\ge (v(G)-1)/2$ implies $\mathrm{src}(G)\le 2$.

Note that the minimum degree requirement in \cref{q:strong-rainbow} is best possible, as below this minimum degree a graph need not be connected. 
On the other hand, the cyclic graph~$H$ showing \cref{q:dirac} is sharp has $\mathrm{src}(H)\ge \mathrm{rc}(H)\ge3$. We can show that the suggested bound on $\mathrm{src}$ is also best possible by giving a concrete edge-coloring of $H$ with three colors. For this, consider the Hamilton cycle consisting of the edges of length $k$. We color one edge with 1 and the remaining edges alternately with 2 and 3. 
We color every other edge with 1.
Consider now two non-adjacent vertices $x$ and $y$. Let their cyclic distance be $d$, where $k+1\le d\le 2k$. 
If~$d=2k$, then there is one geodesic between $x$ and $y$ and the two edges have different colors by construction.
If $d=k+r$, where $r<k$, then there are two geodesic paths between $x$ and $y$ such that one edge has cyclic length $k$. 
One of these length $k$ edges has color different from 1, and therefore forms a rainbow path with the edge of length $r$, which is colored with 1. 

\section*{Acknowledgments}
All three authors were supported by ERC Advanced Grants ``GeoScape'', no.~882971 and ``ERMiD'', no.~101054936.

\end{document}